\documentclass{amsart} 
\usepackage[paperheight=11in, 
    paperwidth=8.5in,
    outer=1.25in,
    inner=1.25in,
    bottom=1.4in,
    top=1.4in]{geometry} 
\usepackage[dvipsnames]{xcolor}
\usepackage{amsmath, amsxtra, amsthm, amsfonts, amssymb}
\usepackage{mathrsfs}
\usepackage{graphicx}
\usepackage{url}
\usepackage{bbm}
\usepackage{tikz-cd}
\usetikzlibrary{fit}

\usepackage[T1]{fontenc}
\usepackage{enumitem}
\usepackage{float}

 \usepackage[hidelinks,backref=page]{hyperref} 
 \hypersetup{
    colorlinks,
    citecolor=magenta,
    filecolor=magenta,
    linkcolor=blue,
    urlcolor=black
}

\usepackage{cleveref}

\newcommand{\RR}{\mathbb R}
\newcommand{\CC}{\mathbb C}

\newcommand{\ZZ}{\mathbb Z}
\newcommand{\be}{\mathbf{e}}
\newcommand{\M}{\mathrm{M}}
\newcommand{\kk}{\mathbbm{k}}
\renewcommand{\P}{\mathrm{P}}
\newcommand{\rk}{\operatorname{rk}}

\theoremstyle{definition}
\newtheorem{thm}{Theorem}[section]

\newtheorem{prop}[thm]{Proposition}
\newtheorem{defn}[thm]{Definition}

\newtheorem{eg}[thm]{Example}
\newtheorem{rem}[thm]{Remark}

\newtheorem{?}[thm]{Question}

\title{Induced Lorentzian and volume polynomials}

\author[Christopher Eur]{Christopher Eur}
\address{Department of Mathematical Sciences, Carnegie Mellon University, Pittsburgh, Pennsylvania 15213}
\email{ceur@andrew.cmu.edu}

\author[Nutan Nepal]{Nutan Nepal}
\address{Department of Mathematics, North Carolina State University, Rayleigh, North Carolina 27695}
\email{nnepal2@ncsu.edu}

\author[Daniel Qin]{Daniel Qin}
\address{Department of Mathematics, University of California, Davis, Davis, California 95616}

\email{dcqin@ucdavis.edu}
\begin{document}
\begin{abstract}
Suppose one has a party of $m$ people, whose expertise collectively covers $n$ topics.  Given a subset $T$ of the topics, one wishes to form a panel of $|T|$ people from the party such that $T$ can be covered by assigning a distinct topic to each panel member with the expertise.  We show that the numbers of such panels, as $T$ varies, form a Lorentzian polynomial.  We achieve this by showing that a certain linear operator on polynomials, which we call the ``inducing operator'' for its connection to induced (poly)matroids, preserves Lorentzian polynomials and realizable volume polynomials.
\end{abstract}
\maketitle

\vspace{-10pt}

\section{Introduction}

For a positive integer $m$, let $[m] = \{1, \dotsc, m\}$, and let $\mathcal S = (S_1, \dotsc, S_n)$ be a finite sequence of subsets of $[m]$.
For a subset $T \subseteq [n]$, consider the following matching statistic:
\begin{align*}
C_T(\mathcal S) := |\{B \subseteq [m] : \text{there is a bijection $\mu\colon T \to B$ with $\mu(t) \in S_t$ for all $t\in T$}\}|.
\end{align*}
We show that these quantities display the following log-concavity property.  For a nonnegative integer $r$, let $\binom{[n]}{r}$ be the set of cardinality-$r$ subsets of $[n]$, and for $T\subseteq [n]$, write $y^T = \prod_{t\in T} y_t$.

\begin{thm}\label{thm:main1}
For any integer $0\leq r \leq m$, the polynomial $f_{\mathcal S, r} := \sum_{T \in \binom{[n]}{r}}  C_T(\mathcal S) y^T$ is \emph{Lorentzian} in the sense of \cite{BH20}.  More strongly, it is a \emph{realizable volume polynomial} over any field in the sense of \cite{GHMSW2025volpres}.
\end{thm}

See \cite{Huh2026VolumePolynomials} for a survey of Lorentzian polynomials and volume polynomials, as well as inequalities that their coefficients satisfy.
In particular, the theorem implies that the quantities $C_T(\mathcal S)$ satisfy variations of the Alexandrov--Fenchel inequality and the reverse Khovanskii--Teissier inequality. We do not know elementary proofs of such inequalities for the statistic $C_T(\mathcal S)$.

\smallskip
We caution that a subset $B$ is counted towards $C_T(\mathcal S)$ only once even if the number of bijections satisfying the condition is greater than one; the alternative statistic is rather straightforward to analyze (Example~\ref{eg:AS}). As a result, a partial derivative of $f_{\mathcal S,r}$ is usually not another polynomial constructed with a similar matching statistic. Therefore, the common inductive strategy involving partial derivatives for showing that a polynomial is Lorentzian does not seem useful in our case. We deduce the theorem by introducing a linear operator on polynomials that preserve Lorentzian polynomials and realizable volume polynomials, as follows. Let $\Gamma_{\mathcal S}$ be a bipartite graph on vertices $[m]\sqcup [n]$ where a vertex $i\in [m]$ is connected to a vertex $j\in [n]$ if and only if $i\in S_j$.

\begin{defn}
We say a pair of $\alpha \in \ZZ_{\geq 0}^m$ and $\beta \in \ZZ_{\geq 0}^n$ admits an \emph{$\mathcal S$-matching} if there exists a weighting $w\colon E(\Gamma_{\mathcal S}) \to \ZZ_{\geq 0}$ of the edges of $\Gamma_{\mathcal S}$ such that, for all $i\in[m] $ and $j\in [n]$, the sum of the weights of the edges incident to $i$ is equal to $\alpha_i$, and to $j$ is equal to $\beta_j$. We write $\alpha\overset{\mathcal S}{\longleftrightarrow}\beta$ in this case.
\end{defn}

\noindent\textbf{Notation.} For $\alpha \in \ZZ_{\geq 0}^m$, we write $\alpha! = \alpha_1 ! \dotsb \alpha_m!$ and $x^\alpha = x_1^{\alpha_1}\dotsb x_m^{\alpha_m}$ and $x^{[\alpha]}=x^\alpha/\alpha!$.  

\begin{thm}\label{thm:main2}
The linear operator $\mathcal I_{\mathcal S}\colon \RR[x_1, \dotsc, x_m] \to \RR[y_1, \dotsc, y_n]$ defined on the monomials by 
\[\mathcal I_{\mathcal S}\left(x^{[\alpha]}\right) := \sum_{\beta: \alpha\overset{\mathcal S}\longleftrightarrow\beta} y^{[\beta]}
\]
preserves Lorentzian polynomials and realizable volume polynomials over a field.
\end{thm}

The operator $\mathcal I_{\mathcal S}$ does not always preserve \emph{stable polynomials} in the sense of \cite{BB09}; see Example~\ref{eg:notstable}. We prove Theorem~\ref{thm:main2} in Section~\ref{sec:symbol} by showing that the symbol of the operator $\mathcal I_{\mathcal S}$ is a realizable volume polynomial over any field.

\begin{proof}[Proof of Theorem~\ref{thm:main1} from Theorem~\ref{thm:main2}]
The polynomial $\sum_{T \in \binom{[n]}{r}}  C_T(\mathcal S) y^T$ is the result of taking the multi-affine part (i.e.\ the terms with square-free monomials) of the polynomial $\mathcal I_{\mathcal S}(\sum_{B\in \binom{[m]}{r}} x_B)$. The elementary symmetric polynomial $\sum_{B\in \binom{[m]}{r}} x_B$ is a realizable volume polynomial over any field because it is over any infinite field $\kk$, witnessed by the arrangement Schubert variety of a generic $r$-dimensional subspace of $\kk^m$ \cite[Example 3.5]{Huh2026VolumePolynomials}, and taking field extensions {does} not change the set of realizable volume polynomials \cite[Proposition 2.10]{GHMSW2025volpres}. The desired conclusion now follows from Theorem~\ref{thm:main2} and that taking the multi-affine part preserves Lorentzian polynomials and realizable volume polynomials \cite[Corollary~3.5]{BH20}, \cite[Corollary~3.3]{GHMSW2025volpres}.
\end{proof}

One can similarly deduce Lorentzian properties about variations of the matching statistic $C_T(\mathcal S)$; see Remark~\ref{rem:rest} and Remark~\ref{rem:CM}.  We illustrate the notions introduced so far in an example for the convenience of the reader.

\begin{eg}
Consider the sequence $\mathcal{S}=(\{1,2,3,4\},\{2,3\},\{3,4\})$. Two distinct $\mathcal{S}$-matchings that witness $(0,2,2,1)\overset{\mathcal{S}}{\longleftrightarrow}(2,2,1)$ are illustrated below. In both $\mathcal{S}$-matchings, \textcolor{NavyBlue}{nonhighlighted} edges are weight zero and edges are thickened to simulate weight. The highlighted edges on the \textcolor{BurntOrange}{left} are all given weight one. On the \textcolor{Orchid}{right}, edge weights are labelled.
\begin{figure}[H]
\centering
\begin{tikzpicture}
    \draw [BurntOrange, line width=1mm, opacity = 0.65] (0,1) --(3,1);
    \draw [BurntOrange, line width=1mm, opacity = 0.65] (0,1) --(3,2);
    \draw [BurntOrange, line width=1mm, opacity = 0.65] (0,-1) -- (3,2);
    \draw [BurntOrange, line width=1mm, opacity = 0.65] (0,0) -- (3,0);
    \draw [BurntOrange, line width=1mm, opacity = 0.65] (0,0) -- (3,1);
    \node [BurntOrange, circle, fill, opacity = 0.65, inner sep = 2pt] at (0,0){};
    
    \node [BurntOrange, circle, fill, opacity = 0.65, inner sep = 2pt] at (3,0){};
    \node [BurntOrange, circle, fill, opacity = 0.65, inner sep = 2pt] at (0,1){};
    \node [BurntOrange, circle, fill, opacity = 0.65, inner sep = 2pt] at (3,1){};
    \node [BurntOrange, circle, fill, opacity = 0.65, inner sep = 2pt] at (3,2){};
    
    \node [BurntOrange, circle, fill, opacity = 0.65, inner sep = 2pt] at (0,-1){};
    \draw [NavyBlue,thick] (0,2) -- (3,2);
    \draw [NavyBlue,thick] (0,1) -- (3,1);
    \draw [NavyBlue,thick] (0,1) -- (3,2);
    \draw [NavyBlue,thick] (0,0) -- (3,2);
    \draw [NavyBlue,thick] (0,0) -- (3,1);
    \draw [NavyBlue,thick] (0,0) -- (3,0);
    \draw [NavyBlue,thick] (0,-1) -- (3,2);
    \draw [NavyBlue,thick] (0,-1) -- (3,0);
    \node [NavyBlue, circle, fill, inner sep=1.5pt] at (0,0) {};
    \node [NavyBlue, circle, fill, inner sep=1.5pt] at (0,1) {};
    \node [NavyBlue, circle, fill, inner sep=1.5pt] at (0,2) {};
    \node [NavyBlue, circle, fill, inner sep=1.5pt] at (3,0) {};
    \node [NavyBlue, circle, fill, inner sep=1.5pt] at (3,1) {};
    \node [NavyBlue, circle, fill, inner sep=1.5pt] at (3,2) {};
    \node [NavyBlue, circle, fill, inner sep=1.5pt] at (0,-1) {};
    \node [left] at (0,-1) {\textcolor{NavyBlue}{$4$}};
    \node [left] at (0,0) {\textcolor{NavyBlue}{$3$}};
    \node [left] at (0,1) {\textcolor{NavyBlue}{$2$}};
    \node [left] at (0,2) {\textcolor{NavyBlue}{$1$}};
    \node [right] at (3,0) {\textcolor{NavyBlue}{$S_3$}};
    \node [right] at (3,1) {\textcolor{NavyBlue}{$S_2$}};
    \node [right] at (3,2) {\textcolor{NavyBlue}{$S_1$}};
    \draw [Orchid, line width=2mm, opacity = 0.65] (6,1) --(9,2);
    \draw [Orchid, line width=2mm, opacity = 0.65] (6,0) -- (9,1);
    \draw [Orchid, line width=1mm, opacity = 0.65] (6,-1) -- (9,0);
    \node [Orchid, circle, fill, opacity = 0.65, inner sep = 2.5pt] at (6,0){};
    
    \node [Orchid, circle, fill, opacity = 0.65, inner sep = 2.5pt] at (6,1){};
    \node [Orchid, circle, fill, opacity = 0.65, inner sep = 2.5pt] at (9,1){};
    \node [Orchid, circle, fill, opacity = 0.65, inner sep = 2.5pt] at (9,2){};
    
    \node [Orchid, circle, fill, opacity = 0.65, inner sep = 2pt] at (6,-1){};
    
    \node [Orchid, circle, fill, opacity = 0.65, inner sep = 2pt] at (9,0){};
    \draw [NavyBlue,thick] (6,2) -- (9,2);
    \draw [NavyBlue,thick] (6,1) -- (9,1);
    \draw [NavyBlue,thick] (6,1) -- (9,2);
    \draw [NavyBlue,thick] (6,0) -- (9,2);
    \draw [NavyBlue,thick] (6,0) -- (9,1);
    \draw [NavyBlue,thick] (6,0) -- (9,0);
    \draw [NavyBlue,thick] (6,-1) -- (9,2);
    \draw [NavyBlue,thick] (6,-1) -- (9,0);
    \node [NavyBlue, circle, fill, inner sep=1.5pt] at (6,0) {};
    \node [NavyBlue, circle, fill, inner sep=1.5pt] at (6,1) {};
    \node [NavyBlue, circle, fill, inner sep=1.5pt] at (6,2) {};
    \node [NavyBlue, circle, fill, inner sep=1.5pt] at (9,0) {};
    \node [NavyBlue, circle, fill, inner sep=1.5pt] at (9,1) {};
    \node [NavyBlue, circle, fill, inner sep=1.5pt] at (9,2) {};
    \node [NavyBlue, circle, fill, inner sep=1.5pt] at (6,-1) {};
    \node [left] at (6,-1) {\textcolor{NavyBlue}{$4$}};
    \node [left] at (6,0) {\textcolor{NavyBlue}{$3$}};
    \node [left] at (6,1) {\textcolor{NavyBlue}{$2$}};
    \node [left] at (6,2) {\textcolor{NavyBlue}{$1$}};
    \node [right] at (9,0) {\textcolor{NavyBlue}{$S_3$}};
    \node [right] at (9,1) {\textcolor{NavyBlue}{$S_2$}};
    \node [right] at (9,2) {\textcolor{NavyBlue}{$S_1$}};
    \node at (7,-0.35) {\textcolor{Orchid}{$1$}};
   \node[white,circle, fill, inner sep=0.05pt] at (7,1.65) {\textcolor{Orchid}{$2$}};
    \node[white,circle, fill, inner sep=0.05pt] at (7,0.65) {\textcolor{Orchid}{$2$}};
    \end{tikzpicture}
    \caption{Two distinct $\mathcal{S}$-matchings that witness $(0,2,2,1)\overset{\mathcal{S}}{\longleftrightarrow}(2,2,1).$}
    \end{figure}
    Continuing the example, one computes that $\mathcal{I}_\mathcal{S}(x_1x_2+x_1x_3+x_2x_3+x_1x_4+x_2x_4+x_3x_4)=6y_1^{[2]}+y_2^{[2]}+y_3^{[2]}+5y_1(y_2+y_3)+3y_2y_3.$ In particular, $C_{\{1,2\}}(\mathcal{S})=C_{\{1,3\}}(\mathcal{S})=5$ and $C_{\{2,3\}}(\mathcal{S})=3.$ Note that there are six distinct bijections $\mu\colon\{1,2\}\rightarrow B$ but five sets $B$ which pair with $\{1,2\}$. Namely, there are 2 ways to match $\{2,3\}$ with $\{S_1,S_2\}.$
\end{eg}

We call the operator $\mathcal I_{\mathcal S}$ the \emph{inducing operator} of $\mathcal S$ because we view it as the polynomial analogue of the construction in matroid theory called \emph{induced (poly)matroids}, as we now explain.
A \emph{polymatroid} $\P$ on $[m]$ is a {function} $\operatorname{rk}_{\P}\colon 2^{[m]} \to \ZZ_{\geq 0}$ satisfying
\begin{enumerate}[label = (\roman*)]
\item $\operatorname{rk}_\P(\emptyset) = 0$,
\item if $I\subseteq J \subseteq [m]$ then $\operatorname{rk}_\P(I) \leq \operatorname{rk}_\P(J)$, and
\item for all subsets $I$ and $J$ of $[m]$, one has $\operatorname{rk}_\P(I) + \operatorname{rk}_\P(J) \geq \operatorname{rk}_\P(I\cup J) + \operatorname{rk}_\P(I\cap J)$.
\end{enumerate}
We call $\operatorname{rk}_\P$ the \emph{rank function} of $\P$, and define the \emph{base polytope} $Q(\P)$ of $\P$ by
\[
Q(\P) := \Big\{x \in \RR^m : \sum_{i\in I} x_i \leq \operatorname{rk}_\P(I) \text{ for all $I\subseteq [m]$}\quad \text{and}\quad \sum_{i\in [m]} x_i = \operatorname{rk}_\P([m])\Big\}.
\]
A \emph{matroid} on $[m]$ is a polymatroid whose rank function further satisfies $\operatorname{rk}_\P(I) \leq |I|$ for all $I \subseteq [m]$.
We point to \cite{Welsh} as a standard reference on matroids and polymatroids.
Following \cite[Chapter~8.2]{Welsh}, from a polymatroid $\P$ on $[m]$ and a sequence of subsets $\mathcal S = (S_1, \dotsc, S_n)$ of $[m]$, one can construct a new polymatroid $\mathcal I_{\mathcal S}(\P)$ on $[n]$, called the \emph{induced polymatroid}, by setting its rank function to be $\operatorname{rk}_{\mathcal I_{\mathcal S}(\P)}(I) := \operatorname{rk}_\P(\bigcup_{i\in I}S_i)$ for all $I\subseteq [n]$.  By setting the rank function to be $I\mapsto \min\{|I|, \operatorname{rk}_\P(\bigcup_{i\in I}S_i)\}$ instead, one obtains the \emph{induced matroid}.

\smallskip
The operator $\mathcal I_{\mathcal S}$ gives rise to induced polymatroids as follows.
If $f  = \sum_{\alpha} c_\alpha x^\alpha \in \RR[x_1, \dotsc, x_m]$ is Lorentzian, then its \emph{support}, i.e.\ $\operatorname{supp}(f) = \{\alpha : c_\alpha \neq 0\}$, is always the set $Q(\P) \cap \ZZ^m$ of the lattice points of the base polytope of some polymatroid $\P$ \cite[Theorem~2.25]{BH20}.
We say simply that the support of $f$ is the polymatroid $\P$ in this case.
The Hall--Rado theorem for polymatroids, recalled here as Proposition~\ref{prop:HR}, implies that 
\begin{equation*}\tag{$*$}\label{eq:induced}
\text{$\operatorname{supp}(f) = Q(\P) \cap \ZZ^m$ $\implies$ $\operatorname{supp}(\mathcal I_{\mathcal S}(f)) = Q(\mathcal I_{\mathcal S}(\P)) \cap \ZZ^n$}.
\end{equation*}
The operator $\mathcal I_{\mathcal S}$ is not the only operator satisfying the property \eqref{eq:induced} however; for example, the change of variables $x_i \mapsto \sum_{j: S_j \ni i} y_j$ also has the property \eqref{eq:induced}.
We point out how $\mathcal I_{\mathcal S}$ is the "center point" among these in Proposition~\ref{prop:1param}.

\section{The symbol of $\mathcal I_{\mathcal S}$}\label{sec:symbol}

We use the ``symbol theory'' originally developed in \cite{BB09} in the context of stable polynomials to prove Theorem~\ref{thm:main2}.
For $\kappa \in \ZZ_{\geq 0}^m$, let $\RR[x_1, \dotsc, x_m]_\kappa$ denote the set of polynomials whose support is contained in the box $[0,\kappa_1] \times \dotsb \times [0, \kappa_m]$.  For $\alpha\in \ZZ_{\geq 0}^m$ we write $\alpha \leq \kappa$ if $\alpha_i \leq \kappa_i$ for all $i \in [m]$.  The \emph{symbol} of a linear operator $T\colon  \RR[x_1, \dotsc, x_m]_\kappa \to \RR[y_1, \dotsc, y_n]$ is defined as
\[
sym_T(y,u) := T\big((x_1+u_1)^{\kappa_1}\dotsb (x_m+u_m)^{\kappa_m}\big) \in \RR[y_1,\ldots,y_n,u_1,\ldots,u_m]
\]
where $T$ acts trivially on $u$, or equivalently, $sym_T(y,u) = \kappa!\sum_{0\leq \alpha\leq  \kappa} T(x^{[\alpha]}){u}^{[\kappa-\alpha]}$.

\begin{prop}\label{prop:symbol}
Let $T$ be a linear operator $\RR[x_1, \dotsc, x_m]_\kappa\to \RR[y_1, \dotsc, y_n]$ that preserves homogeneous polynomials.
\begin{itemize}
\item \cite[Theorem~3.2]{BH20} If the symbol of $T$ is Lorentzian, then $T$ preserves Lorentzian polynomials.
\item \cite[Theorem~1.12]{GHMSW2025volpres} If the symbol of $T$ is a realizable volume polynomial over a field $\kk$, then $T$ preserves realizable volume polynomials over $\kk$.
\end{itemize}
\end{prop}

We now fix an arbitrary $\kappa \in \ZZ_{\geq 0}^m$ for the rest of the section, and consider $\mathcal I_{\mathcal S}$ as an operator on $\RR[x_1, \dotsc, x_m]_\kappa$.
We compute the symbol of $\mathcal I_{\mathcal S}$ to be
\[
sym_{\mathcal I_{\mathcal S}}(y,u) 
= \kappa! \sum_{0\leq \alpha\leq \kappa} \mathcal I_{\mathcal S}\left(x^{[\alpha]}\right) u^{[\kappa-\alpha]}
= \kappa!\sum_{0\leq \alpha \leq \kappa} \sum_{\beta: \alpha\overset{\mathcal S}\longleftrightarrow\beta} y^{[\beta]}u^{[\kappa-\alpha]}.
\]
Note that the symbol is the exponential generating function of its support, which is necessarily a polymatroid if the symbol is to be Lorentzian.
Our primary goal concerns the linear realizability of this polymatroid.
A polymatroid $\P$ on $[m]$ is said to be \emph{linearly realizable} over a field $\kk$ if there exist $\kk$-vector spaces $V_1, \dotsc, V_m$ and a subspace $L \subseteq V_{[m]} = \bigoplus_{i\in [m]} V_i$ such that
\[
\operatorname{rk}_\P(A) = \dim \Big(\text{image of $L$ under the projection $\pi_A\colon V_{[m]} \to V_{A}$}\Big) \quad\text{for all $A\subseteq [m]$}
\]
where $V_A := \bigoplus_{j\in A} V_j$.
It is straightforward to verify that the induced polymatroid $\mathcal I_{\mathcal S}(\P)$ is then linearly realizable over the same field $\kk$ by the image of $L$ under the map $V_{[m]} \to V_{S_1} \oplus \dotsb \oplus V_{S_n}$.

\begin{thm}\label{thm:symbol}
The symbol $sym_{\mathcal I_{\mathcal S}}$ of $\mathcal I_{\mathcal S}$ is the exponential generating function of the lattice points of the base polytope of a polymatroid on $[m]\sqcup [n]$ that is linearly realizable over any field.
\end{thm}

This theorem implies our main result Theorem~\ref{thm:main2} as follows.

\begin{proof}[Proof of Theorem~\ref{thm:main2}]
\cite[Corollary~1.4 \& Remark~3.21]{EL2024intersection} states that the exponential generating function of a linearly realizable polymatroid over $\kk$ is a realizable volume polynomial over $\kk$.
Theorem~\ref{thm:symbol} thus implies that $sym_{\mathcal I_{\mathcal S}}$ is a realizable volume polynomial over any field.  Realizable volume polynomials are also Lorentzian \cite[Theorem~4.6]{BH20}.  Proposition~\ref{prop:symbol} now implies the theorem.
\end{proof}

We prepare for the proof of Theorem~\ref{thm:symbol} with the following variant of the Hall--Rado theorem for polymatroids.

\begin{prop}\label{prop:HR}\cite[Theorem 2]{McD75}
Let $\P$ be a polymatroid on a finite set $E$, and let $\mathcal E = (E_1, \dotsc, E_\ell)$ be a sequence of subsets of $E$.
Then, for $\delta \in \ZZ_{\geq 0}^\ell$, we have that $\delta \in Q(\mathcal I_{\mathcal E}(\P))$ if and only if there exists $\gamma\in {Q(\P) \cap} \ZZ_{\geq 0}^E$ such that the pair $(\gamma,\delta) \in \ZZ_{\geq 0}^E \times \ZZ_{\geq 0}^\ell$ admits {an} $\mathcal E$-matching.
\end{prop}

\begin{proof}[Proof of Theorem~\ref{thm:symbol}]
We first identify the polymatroid.
For an integer $r$ and a finite set $[N]$, let $\Delta_{[N]}^r$ be the polymatroid on $[N]$ whose rank function is $\rk_\P(A) = r$ for all nonempty $A\subseteq [N]$, or equivalently, its base polytope is $Q(\Delta^r_{[N]}) = \{ x\in \RR_{\geq 0}^{[N]} : \sum_{i\in {[N]}} x_i = r\}$.  It is linearly realizable over any field $\kk$ by the diagonal embedding of $\kk^r$ in $(\kk^r)^{\oplus [N]}$.
Let $E$ be the set of edges of the bipartite graph $\Gamma_{\mathcal S}$ on $[m]\sqcup [n]$, and for each vertex $v\in [m]\sqcup [n]$, let $E_v$ be the set of edges incident to $v$.
Let $\P$ be the polymatroid on $E$ defined as the direct sum $\Delta^{\kappa_1}_{E_{1\in [m]}} \oplus \dotsb \oplus \Delta^{\kappa_m}_{E_{m\in [m]}}$.
Here we wrote $E_{i\in[m]}$ for $E_i$ to avoid possible confusion with $E_{i\in [n]}$.
Lastly, let $\mathcal E$ be the sequence $(E_v)_{v\in [m]\sqcup [n]}$, and consider the induced polymatroid $\mathcal I_{\mathcal E}(\P)$ on $[m]\sqcup [n]$, which is linearly realizable over any field $\kk$ because $\P$ is.

We claim that a pair $(\alpha,\beta)\in \ZZ_{\geq 0}^m \times \ZZ_{\geq 0}^n$ with $0\leq \alpha\leq \kappa$ admits {an} $\mathcal S$-matching if and only if $(\kappa-\alpha, \beta) \in Q(\mathcal I_{\mathcal E}(\P))$.
Since $sym_{\mathcal I_{\mathcal S}}$ is the exponential generating function of the set of such $(\kappa-\alpha,\beta)$, the theorem follows. Proposition~\ref{prop:HR} implies that $(\kappa-\alpha, \beta) \in Q(\mathcal I_{\mathcal E}(\P))$ if and only if there exists $\gamma \in Q(\P) \subset \ZZ_{\geq 0}^E$ such that $\gamma$ admits {an} $\mathcal E$-matching to $(\kappa-\alpha, \beta)$. The edge weighting of $\Gamma_{\mathcal E}$ witnessing such {an} $\mathcal E$-matching can be viewed as a function $\mathbf w\colon E \to \ZZ_{\geq 0}^2$, $e \mapsto (w_e, w'_e)$ on the edges of $\Gamma_{\mathcal S}$ such that:
\begin{itemize}
\item $(w_e + w'_e)_{e\in E} \in Q(\P)$, 
\item for every $i\in [m]$, the sum of $w_e$ over all $e\in E$ incident to $i$ is $\kappa_i-\alpha_i$, and
\item for every $j\in [n]$, the sum of $w'_e$ over all $e\in E$ incident to $j$ is $\beta_j$.
\end{itemize}
Since $\P$ is the direct sum of $\Delta^{\kappa_i}_{E_{i\in [m]}}$, we have $Q(\P) = \prod_{i\in [m]} Q(\Delta^{\kappa_i}_{E_{i\in [m]}})$, and thus $(w_e + w'_e)_{e\in E} \in Q(\P)$ if and only if $\sum_{e\in E_{i\in [m]}} (w_e + w'_e) = \kappa_i$ for every $i\in [m]$.  Hence, the last two properties of $\mathbf w$ imply that the edge weighting $w'\colon E \to \ZZ_{\geq 0}$, $e\mapsto w'_e$ witnesses that $(\alpha,\beta)$ admits {an} $\mathcal S$-matching.
Conversely, if $(\alpha,\beta)$ admits {an} $\mathcal S$-matching witnessed by an edge weighting $w'\colon E \to \ZZ_{\geq 0}$, $e\mapsto w'_e$, then we may extend $w'$ to a function $\mathbf w\colon E \to \ZZ_{\geq 0}^2$ satisfying the three bulleted properties above by setting $(w_e)_{e\in E}$ to be any nonnegative integer vector satisfying the second bulleted property.
\end{proof}

\begin{rem}\label{rem:rest}
Let $E$ be the edges of the bipartite graph $\Gamma_{\mathcal S}$.
Given a function $r\colon E \to \ZZ_{\geq 0}$, we say that a pair $(\alpha, \beta)\in \ZZ_{\geq 0}^m \times \ZZ_{\geq 0}^n$ admits {an} \emph{$r$-restricted $\mathcal S$-matching} if it admits {an} $\mathcal S$-matching witnessed by an edge weighting $w\colon  E\to \ZZ_{\geq 0}$ satisfying $w_e \leq r_e$ for all $e\in E$.  For any $r$ and $\kappa$, the proof of Theorem~\ref{thm:symbol} modifies straightforwardly to show that the symbol of the analogue of the operator $\mathcal I_{\mathcal S}$ for $r$-restricted $\mathcal S$-matchings is also a realizable volume polynomial over any field.  The only required change is to replace the polymatroid $\Delta^{\kappa_i}_{E_{i\in [m]}}$ by the polymatroid that is linearly realizable by the image of the map $\kk^{\kappa_i}\to \bigoplus_{e\in E_{i\in [m]}} \kk^{r_e}$ where the map $\kk^{\kappa_i} \to \kk^{r_e}$ for every $e\in E$ can be taken to be any map of full rank.
\end{rem}

A polynomial $f(x_1, \dotsc, x_n)$ is \emph{stable} if it has no roots in the upper half-plane of $\CC^n$. Not all basis generating polynomials of (poly)matroids have such a half-plane property, 
and classifying such (poly)matroids is an open problem initiated in \cite{COSW2004halfplane}.
For example, basis generating functions of certain transversal matroids are not stable~\cite{CW2006Rayleigh}. In agreement, the operator $\mathcal I_{\mathcal S}$ often does not preserve stability.

\begin{eg}\label{eg:notstable}
Let $\mathcal{S}=(123,123,12)$ be a sequence of subsets of $[3]$.
By noting that every $\beta \in \ZZ_{\geq 0}^3$ with $\beta_1 + \beta_2 + \beta_3 = 3$  except ${(0,0,3)}$ admits {an} $\mathcal S$-matching to $(1,1,1)$, one computes that
\[
\mathcal I_{\mathcal S}(x_1x_2x_3) = \frac{1}{6}((y_1+y_2+y_3)^3-y_3^3).\]
We have a factorization
$$(y_1+y_2+y_3)^3-y_3^3=(y_1+y_2)((y_1+y_2+y_3)^2+(y_1+y_2+y_3)y_3+y_3^2)=(y_1+y_2)y_3^2(v^2+v+1)$$ where $v = (y_1+y_2+y_3)/y_3.$
Note that the roots of $v^2+v+1$ lie in the upper half plane of $\CC$.  Using this, we can construct a solution to $I_{\mathcal S}(x_1x_2x_3) = 0$ such as $y_3 = e^\frac{\pi i}{12}$ and $y_1=y_2=\frac{1}{2}(e^{\frac{3\pi i}{4}} - e^{\frac{\pi i}{12}})$, which witnesses that $\mathcal I_{\mathcal S}(x_1x_2x_3)$ is not stable.
\end{eg}

\begin{rem}\label{rem:CM}
We note the following generalization of Theorem~\ref{thm:main1}, which can be deduced from Theorem~\ref{thm:main2} similarly.  For a matroid $\M$ on $[m]$ of rank $r$, and a subset $T\subseteq [n]$, let
\begin{align*}
C_T(\M,\mathcal S) := |\{B \text{ a basis of }\M : \text{there is a bijection $\mu\colon  T \to B$ with $\mu(t) \in S_t$ for all $t\in T$}\}|.
\end{align*}
Then the polynomial $\sum_{T\in \binom{[n]}{r}} C_T(\M,\mathcal S) y^T$ is Lorentzian, and further a realizable polynomial over a field $\kk$ if the matroid $\M$ is realizable over a field of the same characteristic.
\end{rem}

\section{Inducing operators}

Proposition~\ref{prop:HR} implies that the linear operator has the following support property \eqref{eq:induced}:
\begin{quote}
If the support of a Lorentzian polynomial $f\in \RR[x_1, \dotsc, x_m]$ is a polymatroid $\P$, then the support of the polynomial $\mathcal I_{\mathcal S}(f)$ is the induced polymatroid $\mathcal I_{\mathcal S}(\P)$.
\end{quote}
There are other linear operators with this property \eqref{eq:induced}, such as the following.

\begin{eg}\label{eg:AS}
Let $A_{\mathcal S}$ be the matrix whose columns are $\be_{S_1},\dotsc, \be_{S_n}$, where $\be_S := \sum_{i\in S} \be_i \in \RR^m$.  The linear operator $A_{\mathcal S}\colon  f(x) \mapsto f(A_{\mathcal S}y)$, or equivalently the change of variables operator $x_i \mapsto \sum_{S_j \ni i} y_j$, satisfies the support property \eqref{eq:induced}.  Indeed, it is straightforward to verify that the support of $A_{\mathcal S}(x^\alpha)$ and that of $\mathcal I_{\mathcal S}(x^\alpha)$ coincide for $\alpha \in \ZZ_{\geq 0}^m$.
Combinatorially, the distinction between $\mathcal I_{\mathcal S}$ and $A_{\mathcal S}$ is that for $A_{\mathcal S}$, the coefficient of $y^{[\beta]}$ in $A_{\mathcal S}(x^{[\alpha]})$ is given by the number of edge weightings of $\Gamma_{\mathcal S}$ witnessing that $\alpha \overset{\mathcal S}\longleftrightarrow \beta$, whereas $\mathcal I_{\mathcal S}$ only detects the existence of such an edge weighting.
\end{eg}

More generally, let $A$ be a nonnegative $m\times n$ matrix whose zero-nonzero pattern is the same as that of the matrix whose columns are $\be_{S_1},\dotsc, \be_{S_n}$, where $\be_S := \sum_{i\in S} \be_i \in \RR^m$.  Then the linear operator $A\colon  f(x) \mapsto f(Ay)$ satisfies the support property \eqref{eq:induced}.
The symbol of $A$ is Lorentzian, being a product of nonnegative linear forms.

\begin{prop}\label{prop:1param}
Fix any $\kappa\in \ZZ_{\geq 0}^m$.  Let $\mathbf I_{\mathcal S}$ be the space of linear operators $T\colon \RR[x]_\kappa \to \RR[y]$ with Lorentzian symbol and with property \eqref{eq:induced}.  Then $\mathbf I_{\mathcal S}$ deformation retracts to the point $\mathcal I_{\mathcal S}$.
\end{prop}

\begin{proof}
The symbol of such an operator $T\in \mathbf I_{\mathcal S}$ is of the form
\[
\sum_{\substack{\alpha,\beta: 0\leq \alpha\leq \kappa\\ \text{and } \alpha \overset{\mathcal S}\longleftrightarrow \beta}} c_{\alpha,\beta} y^{[\beta]}u^{[\kappa-\alpha]}
\]
for some $c_{\alpha,\beta} > 0$.  For any $0\leq q\leq 1$, define $T^q$ to be the operator whose symbol is
\[
\sum_{\substack{\alpha,\beta: 0\leq \alpha\leq \kappa\\ \text{and } \alpha \overset{\mathcal S}\longleftrightarrow \beta}} c_{\alpha,\beta}^q y^{[\beta]}u^{[\kappa-\alpha]}.
\]
We have $T^0 = \mathcal I_{\mathcal S}$.  For all $0\leq q\leq 1$, taking $q$-th power of normalized coefficients of a Lorentzian polynomial yields a Lorentzian polynomial \cite[Proposition~3.25]{BH20}, so $T^q \in \mathbf I_{\mathcal S}$.
\end{proof}

Since the symbol of $A_{\mathcal S}$ is Lorentzian, the proof of the proposition provides a quick way to verify that the symbol of $\mathcal I_{\mathcal S}$ is Lorentzian.  However, while the symbol of $A_{\mathcal S}$ is a realizable volume polynomial (over any field), being a product of nonnegative linear forms, one cannot conclude similarly that $sym_{\mathcal I_{\mathcal S}}$ is a realizable volume polynomial.  In fact, we ask:

\begin{?}
For a polynomial $f = \sum_\alpha c_\alpha x^{[\alpha]}$, denote $f^q := \sum_\alpha c_\alpha^q x^{[\alpha]}$.  If $f^1$ is a (realizable) volume polynomial (over a field $\kk$), is $f^q$ a volume polynomial (over $\kk$) for all $0\leq q \leq 1$?
\end{?} 
 
One may seek a different 1-parameter family of operators in $\mathbf I_{\mathcal S}$, consisting of operators whose symbols are volume polynomials, that connects $\mathcal I_{\mathcal S}$ and $A_{\mathcal S}$.  The following example shows that this cannot be achieved by varying the coefficients of the matrix $A$.

\begin{eg}
Let $\mathcal S = (1,2,12)$ be a sequence of subsets of $[2]$.  We have $\mathcal I_{\mathcal S}(x_1x_2) = y_1y_2 + y_1y_3 + y_2y_3 + \frac{y_3^2}{2}$, whereas $A(x_1x_2)$ is of the form $ (ay_1 + by_3)(cy_2+dy_3) = acy_1y_2 + ady_1y_3 + bcy_2y_3 + bdy_3^2$.  For all values of $(a,b,c,d)$, the latter has the property that the product of the coefficients of $y_1y_2$ and $y_3^2$ is equal to the product of the coefficients of $y_1y_3$ and $y_2y_3$.
\end{eg}

The desired family of operators can be achieved using Theorem~\ref{thm:symbol} as follows.
For $\mathcal{S}=(S_1,\ldots,S_n)$ on $[m]$, say $S_i=\left\{j_1^{(i)},\ldots,j_{N_i}^{(i)}\right\}$, define another sequence $\widetilde{\mathcal S}$ of subsets of $[m]$ by
$$\widetilde{\mathcal{S}}=\left(S_1,\ldots,S_n, \left\{j_1^{(1)}\right\},\ldots,\left\{j_{N_1}^{(1)}\right\},\ldots,\left\{j_1^{(n)}\right\},\ldots,\left\{j_{N_n}^{(n)}\right\}\right).$$
Let $N = \sum_{i = 1}^nN_i$.  By Theorem~\ref{thm:symbol}, the symbol $sym_{\mathcal I_{\widetilde{\mathcal S}}}(y_1, \dotsc, y_n, \widetilde y_1, \dotsc, \widetilde y_N, u_1, \dotsc, u_m)$ is a realizable volume polynomial over any field $\kk$.
Let $\boldsymbol y'$ be the vector $(y_1, \dotsc, y_1, y_2, \dotsc, y_2, \dotsc, y_n, \dotsc, y_n)$ where $y_i$ is repeated $N_i$ many times.
For any nonnegative vectors $\boldsymbol a\in \RR^n_{\geq 0}$ and $\boldsymbol b\in \RR_{\geq 0}^N$, the nonnegative linear change of variables $sym_{\mathcal I_{\widetilde{\mathcal S}}}(\boldsymbol a \cdot \boldsymbol y, \boldsymbol b \cdot \boldsymbol{ y'}, \boldsymbol u)$ is a volume polynomial over any field.  Let $T_{\boldsymbol a, \boldsymbol b}$ be the linear operator with symbol $sym_{\mathcal I_{\widetilde{\mathcal S}}}(\boldsymbol a \cdot \boldsymbol y, \boldsymbol b \cdot \boldsymbol{ y'}, \boldsymbol u)$. We find that $T_{\mathbf 1,0} = \mathcal I_{\mathcal S}$ and $T_{0,\mathbf 1} = A_{\mathcal S}$.  Thus, any path $(\mathbf{1},0)\rightsquigarrow(0,\mathbf{1})\subseteq \RR_{\geq 0}^n\times \RR_{\geq 0}^{N}$ yields the desired 1-parameter family.

\begin{rem}
In \cite{KMT}, the authors study the analogue of the induced polymatroid construction for M-convex functions.  There, the key strategy is to consider a bipartite graph as a composition of simple bipartite graphs, called elementary splits and elementary aggregations.  The same strategy cannot be adopted to study $\mathcal I_{\mathcal S}$ because the composition of the operators of the form $\mathcal I_{\mathcal S}$ does not respect composition of bipartite graphs, as the following example shows.
\end{rem}

\begin{eg}
    Take vertex sets $U=\{1,2\}$, $V=\{1_a,1_b,2_a,2_b\}$, and $W=\{a,b\}$, with sequences
    \[
        \mathcal S_1=(\{1\},\{1\},\{2\},\{2\}) \text{ on } U
        \qquad\text{and}\qquad
        \mathcal S_2=(\{1_a,2_a\},\{1_b,2_b\}) \text{ on } V.
    \]
The bipartite graphs $\Gamma_{\mathcal S_1}$ and $\Gamma_{\mathcal S_2}$ are illustrated respectively as the {middle and right parts} of the figure below.  Each of the two consists only of elementary splits and elementary aggregations in the sense of \cite{KMT}. The composition $\Gamma_{\mathcal S_2}\circ\Gamma_{\mathcal S_1}$, obtained by forgetting the bivalent nodes in the middle layer, is the complete bipartite graph $K_{2,2}$ between $U$ and $W$.
    \begin{figure}[H]
        \centering
        \begin{tikzpicture}[scale=0.85]
    \draw [NavyBlue,thick] (0,4) -- (3,4);
    \draw [NavyBlue,thick] (0,1) -- (3,1);
    \draw [NavyBlue,thick] (0,1) -- (3,4);
    \draw [NavyBlue,thick] (0,4) -- (3,1);
    
    \node [NavyBlue,circle, fill, inner sep=1.5pt] at (0,4) {};
    \node [NavyBlue,circle, fill, inner sep=1.5pt] at (0,1) {};
    \node [NavyBlue,circle, fill, inner sep=1.5pt] at (3,4) {};
    \node [NavyBlue,circle, fill, inner sep=1.5pt] at (3,1) {};
    \node [NavyBlue,left] at (0,4) {{$1$}};
    \node [NavyBlue,left] at (0,1) {{$2$}};
    \node [NavyBlue,right] at (3,4) {{$a$}};
    \node [NavyBlue,right] at (3,1) {{$b$}};

    \node at (4,2.5) {{$=$}};
    \draw [NavyBlue,thick] (8,2) -- (5,4);
    \draw [NavyBlue,thick] (8,4) -- (5,4);
    \draw [NavyBlue,thick] (8,1) -- (5,1);
    \draw [NavyBlue,thick] (8,3) -- (5,1);
    
    \node [NavyBlue,circle, fill, inner sep=1.5pt] at (8,1) {};
    \node [NavyBlue,circle, fill, inner sep=1.5pt] at (8,2) {};
    \node [NavyBlue,circle, fill, inner sep=1.5pt] at (8,3) {};
    \node [NavyBlue,circle, fill, inner sep=1.5pt] at (8,4) {};
    \node [NavyBlue,circle, fill, inner sep=1.5pt] at (5,4) {};
    \node [NavyBlue,circle, fill, inner sep=1.5pt] at (5,1) {};
    \node [NavyBlue,right] at (8,1) {{$2b$}};
    \node [NavyBlue,right] at (8,2) {{$1b$}};
    \node [NavyBlue,right] at (8,3) {{$2a$}};
    \node [NavyBlue,right] at (8,4) {{$1a$}};
    \node [NavyBlue,left] at (5,4) {{$1$}};
    \node [NavyBlue,left] at (5,1) {{$2$}};
    \draw [NavyBlue,thick] (11,2) -- (12.5,1);
    \draw [NavyBlue,thick] (11,3) -- (12.5,4);
    \draw [NavyBlue,thick] (11,4) -- (12.5,4);
    \draw [NavyBlue,thick] (11,1) -- (12.5,1);
    
    \node [NavyBlue,circle, fill, inner sep=1.5pt] at (11,1) {};
    \node [NavyBlue,circle, fill, inner sep=1.5pt] at (11,2) {};
    \node [NavyBlue,circle, fill, inner sep=1.5pt] at (11,3) {};
    \node [NavyBlue,circle, fill, inner sep=1.5pt] at (11,4) {};
    \node [NavyBlue,circle, fill, inner sep=1.5pt] at (12.5,4) {};
    \node [NavyBlue,circle, fill, inner sep=1.5pt] at (12.5,1) {};
    \node [NavyBlue,left] at (11,1) {{$2b$}};
    \node [NavyBlue,left] at (11,2) {{$1b$}};
    \node [NavyBlue,left] at (11,3) {{$2a$}};
    \node [NavyBlue,left] at (11,4) {{$1a$}};
    \node [NavyBlue,right] at (12.5,4) {{$a$}};
    \node [NavyBlue,right] at (12.5,1) {{$b$}};
    \node at (9.43,2.5) {{$\longrightarrow$}};
        \end{tikzpicture}
        \caption{The composition visualized}
    \end{figure}

\noindent We find that $\mathcal I_{\mathcal S_2 \circ \mathcal S_1}(x_1x_2)=z_a z_b+z_a^{[2]}+z_b^{[2]}.$
On the other hand, $(\mathcal I_{\mathcal S_2}\circ\mathcal I_{\mathcal S_1})(x_1x_2)=2z_a z_b+z_a^{[2]}+z_b^{[2]}$. The extra coefficient on $z_a z_b$ comes from two distinct intermediate degree vectors on $V$ mediating the $(\mathcal{S}_2\circ\mathcal{S}_1)$-matching $(1,1)\longleftrightarrow(1,1)$, namely $(1,0,0,1)$ and $(0,1,1,0)$.
\end{eg}

\section*{Acknowledgements}
We thank American Institute of Mathematics for hosting the workshop ``The geometry of polynomials in combinatorics and sampling'' during which the first author proposed this project. We also thank the referees for their careful reading and valuable feedback. CE is supported by US NSF grant DMS-2246518. NN is supported by NSF Grant DMS-2452179 and Simons Foundation Travel Support for Mathematicians Award MPS-TSM00007970. DQ is partially supported by US NSF grant CCF-2317280.
{\fontsize{10.5}{11}\selectfont \bibliographystyle{alpha}
  \bibliography{references}  }
\end{document}